\documentclass[12pt]{amsart}
\usepackage{amsmath,amssymb,amsbsy,amsfonts,amsthm,latexsym,mathabx,
            amsopn,amstext,amsxtra,euscript,amscd,stmaryrd,mathrsfs,
            cite,array,mathtools,enumerate}

\usepackage{url}
\usepackage[colorlinks,linkcolor=blue,anchorcolor=blue,citecolor=blue,backref=page]{hyperref}

\usepackage{color}

\usepackage[norefs,nocites]{refcheck}

\usepackage{color}
\usepackage{float}

\hypersetup{breaklinks=true}

\usepackage[english]{babel}

\usepackage{enumitem}

\begin{document}

\newtheorem{theorem}{Theorem}
\newtheorem{lemma}[theorem]{Lemma}
\newtheorem{claim}[theorem]{Claim}
\newtheorem{cor}[theorem]{Corollary}
\newtheorem{prop}[theorem]{Proposition}
\newtheorem{definition}{Definition}
\newtheorem{question}[theorem]{Open Question}
\newtheorem{example}[theorem]{Example}

\numberwithin{equation}{section}
\numberwithin{theorem}{section}

 \newcommand{\F}{\mathbb{F}}
\newcommand{\K}{\mathbb{K}}
\newcommand{\D}[1]{D\(#1\)}
\def\scr{\scriptstyle}
\def\\{\cr}
\def\({\left(}
\def\){\right)}
\def\[{\left[}
\def\]{\right]}
\def\<{\langle}
\def\>{\rangle}
\def\fl#1{\left\lfloor#1\right\rfloor}
\def\rf#1{\left\lceil#1\right\rceil}
\def\le{\leqslant}
\def\ge{\geqslant}
\def\eps{\varepsilon}
\def\mand{\qquad\mbox{and}\qquad}

\def \e{\mathbf{e}}

\newcommand{\commI}[1]{\marginpar{%
\begin{color}{red}
\vskip-\baselineskip %raise the marginpar a bit
\raggedright\footnotesize
\itshape\hrule \smallskip I: #1\par\smallskip\hrule\end{color}}}

\newcommand{\commL}[1]{\marginpar{%
\begin{color}{blue}
\vskip-\baselineskip %raise the marginpar a bit
\raggedright\footnotesize
\itshape\hrule \smallskip L: #1\par\smallskip\hrule\end{color}}}

\newcommand{\Fq}{\mathbb{F}_q}
\newcommand{\Fp}{\mathbb{F}_p}
\newcommand{\Disc}[1]{\mathrm{Disc}\(#1\)}
\newcommand{\Res}[1]{\mathrm{Res}\(#1\)}
\newcommand{\QQ}{\mathbb{Q}}
\newcommand{\ZZ}{\mathbb{Z}}

\newcommand{\p}{\mathfrak{p}}

\renewcommand{\Re}{\mathrm{Re}}
\def\GL{\mathrm{GL}}
%%%%%%%%%%%%%%%%%%%%%%%%%
% Alphabet calligraphie %
%%%%%%%%%%%%%%%%%%%%%%%%%
\def\cA{{\mathcal A}}
\def\cB{{\mathcal B}}
\def\cC{{\mathcal C}}
\def\cD{{\mathcal D}}
\def\cE{{\mathcal E}}
\def\cF{{\mathcal F}}
\def\cG{{\mathcal G}}
\def\cH{{\mathcal H}}
\def\cI{{\mathcal I}}
\def\cJ{{\mathcal J}}
\def\cK{{\mathcal K}}
\def\cL{{\mathcal L}}
\def\cM{{\mathcal M}}
\def\cN{{\mathcal N}}
\def\cO{{\mathcal O}}
\def\cP{{\mathcal P}}
\def\cQ{{\mathcal Q}}
\def\cR{{\mathcal R}}
\def\cS{{\mathcal S}}
\def\cT{{\mathcal T}}
\def\cU{{\mathcal U}}
\def\cV{{\mathcal V}}
\def\cW{{\mathcal W}}
\def\cX{{\mathcal X}}
\def\cY{{\mathcal Y}}
\def\cZ{{\mathcal Z}}

\def\Q{\mathbb{Q}}
\def\Z{\mathbb{Z}}

\newcommand{\sR}{\ensuremath{\mathscr{R}}}
\newcommand{\sDI}{\ensuremath{\mathscr{DI}}}
\newcommand{\DI}{\ensuremath{\mathrm{DI}}}

\newcommand{\Nm}[1]{\mathrm{Norm}_{\,\F_{q^k}/\Fq}(#1)}

\newcommand{\Tr}[1]{\mathrm{Tr}\(#1\)}
\newcommand{\rad}[1]{\mathrm{rad}(#1)}

\newcommand{\Orb}[1]{\mathrm{Orb}\(#1\)}
\newcommand{\aOrb}[1]{\overline{\mathrm{Orb}}\(#1\)}

\renewcommand{\v}{\mathrm{v}}

\title[Distribution  of inversive generator]
{Distribution of short subsequences of inversive congruential pseudorandom 
numbers modulo $2^t$}

 \author[L. M{\'e}rai]{L{\'a}szl{\'o} M{\'e}rai}
\address{L.M.: Johann Radon Institute for
Computational and Applied Mathematics,
Austrian Academy of Sciences, Altenberger Stra\ss e 69, A-4040 Linz, Austria}
\email{laszlo.merai@oeaw.ac.at}

\author[I.~E.~Shparlinski]{Igor E. Shparlinski}
\address{I.E.S.: School of Mathematics and Statistics, University of New South Wales.
Sydney, NSW 2052, Australia}
\email{igor.shparlinski@unsw.edu.au}

\begin{abstract}
 In this paper we study the distribution of very short sequences  of inversive congruential pseudorandom numbers
 modulo   $2^t$. We derive a new bound on exponential sums with such sequences and  use it to  estimate  their discrepancy.
 The technique we use is based on the method of N.~M.~Korobov (1972) of estimating double Weyl sums and a fully explicit 
 form of the Vinogradov  mean value theorem due to K.~Ford (2002), which has never been used in this area and is very likely to
 find further applications. 
 \end{abstract}


\pagenumbering{arabic}

\maketitle

\section{Introduction}

\subsection{Background on the M{\"o}bius tranformation}

Let $t\geq 3$ be an integer and write $\cU_t= \cR_t^*$ for the group of units of the residue ring $\cR_t = \Z/2^t\Z$ modulo $2^t$. 
Then $\#\cU_t=2^{t-1}$. It is often be convenient to identify elements of $\cR_t$ with the corresponding elements of the least residue system modulo $2^t$.

We fix a matrix 
$$
M=\begin{pmatrix} m_{11} & m_{12}\\ m_{21} & m_{22} \end{pmatrix} \in \GL_2(\cR_t)
$$ 
with
\begin{equation}
\label{eq:Matr M}
M %= \begin{pmatrix} m_{11} & m_{12}\\ m_{21} & m_{22} \end{pmatrix} 
\equiv \begin{pmatrix} 1 & 0 \\ 0 &1 \end{pmatrix} 
\quad
\text{or}
\quad
\begin{pmatrix} 0 & 1 \\ 1 &0 \end{pmatrix}
\mod 2.
\end{equation}
%with entries from $\cR_t$, which is nonsingular modulo $2$.

We then consider sequences  generated by iterations of the
{\it M{\"o}bius tranformation\/}
\begin{equation}
\label{eq:MobMap}
M: \ x \mapsto \frac{m_{11} x+ m_{12}}{m_{21} x + m_{22}}
\end{equation}
which, under the condition~\eqref{eq:Matr M},  is always defined over   $\cU_t$, that is, $M: \cU_t \to \cU_t$.

That is  for $u_0\in \cR_t$ we consider the trajectory 
\begin{equation}
\label{eq:Traj}
u_{n} =M\(u_{n-1}\) = M^n\(u_0\),
 \qquad n =  1,2,
\ldots\,,
\end{equation}
generated  by iterations of the M{\"o}bius tranformation~\eqref{eq:MobMap} associated with $M$. 

Assume that the characteristic polynomial of $M$ has two  {\it distinct\/} 
 eigenvalues $\vartheta_1$ and  $\vartheta_2$ from the algebraic closure $\overline \Q_2$ 
 of the field of $2$-adic fractions $\Q_2$.

It is not difficult to prove by induction on $n$ that there is an explicit 
representation of the form
\begin{equation}\label{eq:seq gen}
 u_n= \frac{\gamma_{11} \vartheta_1^n + \gamma_{12} \vartheta_2^n}
 {\gamma_{21} \vartheta_1^n + \gamma_{22} \vartheta_2^n}
\end{equation}
with some coefficients  $\gamma_{ij} \in \overline \Q_2$, $i,j =1,2$.

Here we consider the split case when the eigenvalues  $\vartheta_1, \vartheta_2 \in \Z_2$
are $2$-adic integers, in which case, interpolating,  we also have  $\gamma_{ij} \in \Z_2$, $i,j=1,2$.

It is easy to see that in this case we can assume that  
$$
\gamma_{21} \equiv 1 \bmod 2 \mand  \gamma_{22}\equiv 0 \bmod 2. 
$$ 
Then, defining $g \in \cU_t$ by the equation
$$
g = \vartheta_1/ \vartheta_2
$$
we have $g\in \cR_t$ (recall that $M$ is invertible in $\cR_2$), thus the sequence generated by~\eqref{eq:Traj}, 
the representation~\eqref{eq:seq gen} has the form 
\begin{equation}\label{eq:seq}
 u_n=\frac{a}{g^n-b}+c
\end{equation}
with some  coefficients $a,b,c \in \cR_t$.
Furthermore, it is also easy to see that 
$$
b \equiv 0 \bmod 2. 
$$

\subsection{Motivation}
The sequences~\eqref{eq:Traj} are interesting in their own rights but they have also been 
used as a source of pseudorandom number generation where this sequence is known  as
the {\it inversive generator\/}, for example, see~\cite{Chou} for the period length and~\cite{NiedWint}
for distributional properties. 

More precisely, let $\tau$ be the  multiplicative order of $g$ modulo $2^t$. 
Then $(u_n)$ is a periodic sequence with period length $\tau$, provided that $a$ is odd.

Niederreiter and Winterhof~\cite{NiedWint}, extending the results
of~\cite{NiedShp} from odd prime powers to powers of $2$, obtained nontrivial results for
segments of these sequences of length $N$ 
satisfying 
\begin{equation}
\label{eq:NW range}
\tau \ge N \ge 2^{(1/2 + \eta) t}
\end{equation}
for any fixed $\eta>0$ and sufficiently large $t$.    

Here using very different techniques we  significantly reduce the range \eqref{eq:NW range}
and obtain results which are nontrivial for much shorter segments, namely, 
for
\begin{equation}
\label{eq:MS range}
\tau \ge N \ge 2^{ct^{2/3}}
\end{equation}
for some absolute constant $c>0$.

We also consider this as an opportunity to introduce new techniques into the area of 
pseudorandom number generation which we believe may have more applications 
and  lead to new advances.

\subsection{Our results}

Here we establish upper bounds for the exponential sums
$$
S_h(L,N)=\sum_{n=L}^{L+N-1} \e\(h u_n/2^t \), \qquad 1\leq N\leq \tau, 
$$
where, as usual, we denote $\e(z) = \exp(2\pi i z)$ and, 
as before,  $\tau$ is the  multiplicative order of $g$ modulo $2^t$.

Using the method of Korobov~\cite{Kor} together with  the use of the  Vinogradov  mean value theorem in
the explicit form given by Ford~\cite{Ford}, 
we can estimate $S_h(L,N)$ for the values $N$ in the range~\eqref{eq:MS range}. 

Throughout the paper we always use the parameter 
\begin{equation}
\label{eq:rho}
\rho = \frac{\log N}{t}
\end{equation}
which controls the size of $N$ relative to the modulus $2^t$ on a logarithmic scale.

\begin{theorem}\label{thm:main}
Let $\gcd(g,2)=\gcd(a,2)=1$ and write
$$
g^2=1+w_\beta2^\beta, \quad \gcd(w_\beta,2)=1.
$$
Then for $2^{8\beta} < N \leq \tau$ we have
$$
\left|S_h(L,N) \right| \leq c  N^{1-\eta\rho^2}
$$
where $\rho$ is given by~\eqref{eq:rho}, 
for some absolute constants $c, \eta>0$ uniformly  over all integers $h$ with $\gcd(h,2)=1$.
\end{theorem}

From a sequence $(u_n)$ defined by~\eqref{eq:seq} we derive the \emph{inversive congruential pseudorandom numbers with modulus $2^t$}:
$$ 
u_{L}/2^t, u_{L+1}/2^t, \ldots, u_{L+N-1}/2^t \in[0,1).
$$
The \emph{discrepancy} $D(L,N)$ of these numbers is defined by
$$
D(L,N)=\sup_{J\subset[0,1)}\left|\frac{A(J,N)}{N}-|J| \right|,
$$
where the supremum is taken over all subintervals $J$ of $[0,1)$, $A(N,J)$ is the number of point $u_n/2^t$ in $J$ for $L\leq n<L+N$, 
and $|J|$ is the length of $J$. The   \emph{Erd{\H o}s--Tur{\'a}n inequality} (see~\cite[Theorem~1.21]{DrTi})
  allows us to give an upper bound on the discrepancy $D(L,N)$ in terms of $S_h(L,N)$.

\begin{theorem}\label{thm:discrepancy}
Let $(u_n)$ be as in Theorem~\ref{thm:main} and assume that  $2^{32\beta} < N \leq \tau$. 
Then we have
$$
 D(L,N) \leq c_0  N^{-\eta_0 \rho ^2}
$$
where $\rho$ is given by~\eqref{eq:rho}, 
for some constants $c_0, \eta_0>0$. 
\end{theorem}

Writing 
$$
N^{-\rho ^2} = \exp\(-\frac{(\log N)^3}{t^2} \)
$$
we see that Theorems~\ref{thm:main} and~\ref{thm:discrepancy} are nontrivial in the range~\eqref{eq:MS range}.

\section{Preparation}

\subsection{Notation}
We recall that the notations $U\ll V$, and $V\gg U$ are 
equivalent to the statement that the inequality $|U |\leq cV$ holds with some
absolute constant $c>0$. 
 
We use the notation $\v_2$ to the 2-adic valuation, that is, for non-zero integers $a\in\Z$ we let $\v_2(a)=k$ if $2^k$ is the highest power of 2 which divides $a$, and $\v_2(a/b)=\v_2(a)-\v_2(b)$ for $a,b\neq 0$.

\subsection{Multiplicative order of integers}

The following assertion describes the order of elements modulo   powers of $2$. 

\begin{lemma}\label{lemma:order}
Let $g\neq \pm 1$ be an odd integer and write
$$
g^2=1+w_\beta 2^\beta, \quad \gcd (w_\beta,2)=1.
$$
Then for $s\geq \beta$ the multiplicative order $\tau_s$ of $g$ modulo $2^s$ is $ \tau_s=2^{s-\beta +1}$ and 
\begin{equation}\label{eq:order}
g^{\tau_s}=1+w_s2^{s}, \quad \gcd (w_s,2)=1.
\end{equation}
\end{lemma}
\begin{proof} First we note that $\beta\ge 2$.  We prove~\eqref{eq:order} by induction of $s$.

Clearly, we have~\eqref{eq:order} with $s=\beta$. Furthermore, 
if~\eqref{eq:order} holds for some $s\geq \beta$, then by squaring it we get
$$
g^{2 \tau_s}=1+w_s2^{s+1}+w_s^22^{2s+2}=1+w_{s+1}2^{s+1} , 
$$
with $w_{s+1}=1+w_s2^{s-1} \equiv 1 \bmod 2$. Hence~\eqref{eq:order} also holds with $s+1$
in place of $s$. 
\end{proof}

\subsection{Explicit form of the Vinogradov mean value theorem}

Let $N_{k,n}(M)$ be the number of integral solutions of the system of equations
\begin{align*}
 x_1^j+\ldots+x_k^j&=y_1^j+\ldots+y_k^j, \qquad j =1, \ldots, n,\\
 1\leq x_i&,y_i\leq M,  \qquad i =1, \ldots, k.
\end{align*}

Our application of Lemma~\ref{lem:Kor} below rests on a version of the 
Vinogradov mean value theorem which  gives a precise bound on $N_{k,n}(M)$.
We use its fully explicit 
version given by Ford~\cite[Theorem~3]{Ford}, which we present 
here in   the following weakened and simplified form. 

\begin{lemma}
\label{lem:ford}
For every integer $n\ge 129$ there exists an integer
$k\in[2n^2,4n^2]$ such that for any integer $M\ge 1$ we have
$$
N_{k,n}(M)\le n^{3n^3}M^{2k-0.499n^2}. 
$$
\end{lemma}

We note that the recent striking advances  in the Vinogradov mean value theorem
due to 
Bourgain, Demeter and Guth~\cite{BDG} and Wooley~\cite{Wool}  are not suitable
for our purposes here as they contain implicit constants
that depend on $k$ and $n$, while in our approach   $k$ and $n$
grow together  with $M$.

\subsection{Double exponential sums with polynomials}

Our main tool to bound the exponential sum $S_h(L,N)$ is the following 
result of Korobov~\cite[Lemma~3]{Kor}.

\begin{lemma}\label{lem:Kor}
Assume that
$$
\left| \alpha_\ell -\frac{a_\ell}{q_\ell} \right| \le \frac{1}{q_\ell^2} \mand \gcd(a_\ell,q_\ell)=1, 
$$
for some real $\alpha_\ell$ and  integers $a_\ell,q_\ell$, $\ell=1,\ldots, n$. Then
for the sum
$$
S=\sum_{x,y=1}^M \e\(\alpha_1xy+\ldots+\alpha_n x^ny^n \)
$$
we have 
\begin{align*}
|S|^{2k^2}\leq \(64k^2\log(3Q)\)^{n/2} &
M^{4k^2-2k}N_{k,n}(M)\\
& \prod_{\ell=1}^n \min\left\{M^\ell,\sqrt{q_\ell}+\frac{M^\ell}{\sqrt{q_\ell}} \right\},
\end{align*}
where 
$$
Q=\max\{q_\ell~:~1\le \ell \le n\}.
$$
\end{lemma}

We also need the following simple result which allows us to reduce single sums 
to double sums. 

\begin{lemma}\label{lemma:sum}
 Let $f:\mathbb{R}\rightarrow\mathbb{R}$ be an arbitrary function. Then for any  integers $M,N\ge 1$ and $a\ge 0$, we have
 $$
 \left|\sum_{x=0}^{N-1}\e(f(x)) \right|\leq \frac{1}{M^2}\sum_{x=0}^{N-1}\left|\sum_{y,z=1}^{M}\e(f(x+ayz))\right|+2 a M^2.
 $$
\end{lemma}

\begin{proof} Examining the non-overlapping parts of the sums below, we see that for any positive integers $y$ and $z$
$$
 \left|\sum_{x=0}^{N-1}\e(f(x)) - \sum_{x=0}^{N-1} \e(f(x+ayz))\right|  \le 2ayz.
$$
Hence 
$$
 \left|M^2 \sum_{x=0}^{N-1}\e(f(x)) - \sum_{y,z=1}^{M} \sum_{x=0}^{N-1} \e(f(x+ayz))\right|  \le 2a\sum_{y,z=1}^{M} yz \le 2 a M^4.
$$
Changing the order of summation and using the triangle inequality, 
 the result  follows. 
\end{proof}

\subsection{Sums of binomial coefficients}

We need results of certain sums of binomial coefficients. The first ones are immediate and we leave the proof for the reader.

\begin{lemma}\label{lemma:binom}   Let $n$  be a positive integer. Then 
\begin{enumerate}
 \item for any integer $k\leq n$ we have
 $$
 \sum_{i=k}^n \binom{i}{k}=\binom{n+1}{k+1};
 $$
 \item  for any  polynomial $P(X)\in\ZZ[X]$ of degree $\deg P<n$ we have
$$
\sum_{j=0}^n(-1)^j\binom{n}{j}P(j)=0.
$$
\end{enumerate}
\end{lemma}

\begin{lemma}\label{lemma:multinomial}
For any $n,k$ with $k\leq n$ we have
 $$
 \sum_{\substack{\ell_1+\ldots+\ell_k=n\\ \ell_1,\ldots, \ell_k\geq 1}}\frac{n!}{\ell_1!\ldots\ell_k!}=\sum_{i=0}^k (-1)^{k-i} \binom{k}{i}i^n.
 $$
\end{lemma}
\begin{proof}
As
$$
\sum_{\substack{\ell_1+\ldots+\ell_k=n}}\frac{n!}{\ell_1!\ldots\ell_k!}=k^n
$$
the result follows directly from the inclusion--exclusion principle.
\end{proof}

\section{Proofs of the main results} 

\subsection{Proof of Theorem~\ref{thm:main}}
\label{sec:proof exp}
As
$$
u_{n+L}=\frac{a}{g^{n+L}-b}+c=\frac{ag^{-L}}{g^{n}-bg^{-L}}+c,
$$
we can assume, that $L=0$ and we put
$$
S_h(0,N)=S_h(N).
$$
We can also assume, that $a=1$ and $c=0$. Finally we assume, that
$$
N\geq 2^{6t^{1/2}}
$$
since otherwise the result is trivial, see~\eqref{eq:MS range}.

Define
$$
r=\frac{t\log 2}{\log N} = \rho^{-1} \log 2, 
$$
 where $\rho$ is given by~\eqref{eq:rho}.
First assume, that
$$
r\geq 129
$$
and put
$$
s=\left \lfloor\frac{t}{4r} \right \rfloor \quad \text{and} \quad \kappa=\left \lceil \frac{t}{s}\right \rceil-1.
$$
Then
$$
s>\beta, \quad 2^s \leq N^{1/4}, \quad r\leq \kappa <s, 
$$
if $N$ is large enough. Indeed,
$$
s\geq \frac{t}{4r}-1=\frac{\log N}{4 \log 2}-1\geq 2 \beta -1>\beta \quad \text{and} \quad 2^s\leq 2^{\frac{t}{4r}}=N^{1/4}.
$$
Moreover,
$$
\kappa \geq \frac{t}{s}-1\geq 4r-1\geq r
$$
and
$$
\kappa \leq \frac{t}{s}\leq \frac{(\log N)^2}{36 (\log 2)^2 s}=\frac{t^2}{36 r^2 s}\leq s.
$$

Let $\tau_s$ be the order of $g$ modulo $2^s$. As $s>\beta$, 
$$
 g^{\tau_s}=1+w\cdot 2^s \quad \text{with } \gcd(w,2)=1
$$
by Lemma~\ref{lemma:order}.  
Clearly, for all even $x$, we have
 $$
 \frac{1}{1-x}\equiv 1+x+\ldots + x^{t-1} \bmod 2^t,
 $$
 thus
\begin{align*}
 u_{n\cdot \tau_s}&\equiv \frac{-1}{b-g^{n\cdot \tau_s}}\equiv \frac{-1}{1-(1-b+g^{n\tau_s})}\equiv -\sum_{\ell=0}^{t-1}\(1-b+g^{n\cdot \tau_s}\)^\ell\\
 & \equiv  -\sum_{\ell=0}^{t-1}\(1-b+\(1+w\cdot 2^s \)^n\)^\ell\\
 & \equiv - \sum_{\ell=0}^{t-1}\(2-b+ \sum_{i=1}^n\binom{n}{i}(w\cdot 2^s)^i \)^\ell \bmod 2^t.
 \end{align*}
Define
$$
F_\kappa(n)= \sum_{\ell=0}^{\kappa}(w\cdot 2^s)^\ell\sum_{j=0}^{t-1} \sum_{\nu=1}^j \binom{j}{\nu} (2-b)^{j-\nu} \sum_{\substack{i_1+\ldots+i_\nu=\ell\\ i_1,\ldots, i_\nu\geq 1}}\binom{n}{i_1}\ldots\binom{n}{i_\nu}.
$$
Then
$$
 u_{n\cdot \tau_s}\equiv -F_{\kappa}(n) \bmod 2^t.
$$

The expression $\kappa!F_\kappa(n)$ is a polynomial of $2^sn$ of degree at most $\kappa$. Thus we can define the integers $a_0,\ldots, a_\kappa$ by 
$$
\kappa!F_\kappa(n)=\sum_{\ell=0}^\kappa a_\ell 2^{\ell s} n^\ell.
$$
Then the coefficients satisfy
$$
a_\ell \equiv\frac{\kappa!}{\ell!} w^\ell \sum_{j=1}^{t-1} \sum_{\nu=1}^j \binom{j}{\nu} (2-b)^{j-\nu} \sum_{\substack{i_1+\ldots+i_\nu=\ell\\ i_1,\ldots, i_\nu\geq 1}}\frac{\ell!}{i_1!\ldots i_\nu!} \bmod 2^s.
$$
We have $\v_2(a_\ell)=\v_2(\kappa!/\ell!)$. Indeed, as $w$ is odd and $b$ is even, by 
Lemmas~\ref{lemma:multinomial} and~\ref{lemma:binom} we get
\begin{align*}
\sum_{j=1}^{t-1} &\sum_{\nu=1}^j \binom{j}{\nu} (2-b)^{j-\nu} \sum_{\substack{i_1+\ldots+i_\nu=\ell\\ i_1,\ldots, i_\nu\geq 1}}\frac{\ell!}{i_1!\ldots i_\nu!}\\
& \equiv \sum_{j=1}^{\ell} \sum_{\substack{i_1+\ldots+i_j=\ell\\ i_1,\ldots, i_j\geq 1}}\frac{\ell!}{i_1!\ldots i_j!}
\equiv  \sum_{j=1}^{\ell} \sum_{i=0}^j (-1)^{j-i} \binom{j}{i}i^\ell  \\
&\equiv \sum_{i=0}^\ell(-1)^{i}i^\ell\sum_{j=i}^\ell \binom{j}{i}\equiv\sum_{i=0}^\ell(-1)^{i}i^\ell\binom{\ell+1}{i+1}\\
& \equiv-\sum_{i=1}^{\ell+1}(-1)^{i}\binom{\ell+1}{i}(i-1)^\ell\equiv \binom{\ell+1}{0}(-1)^\ell \equiv1 \bmod 2
\end{align*}
(we note that the last several congruences are actually equations).

Write $\omega_\ell=\v_2(a_\ell)$. Then
$$
\omega_\ell \leq \v_2(\kappa!)\leq \left\lfloor\frac{\kappa}{2} \right\rfloor+\left\lfloor\frac{\kappa}{4} \right\rfloor+\ldots <\kappa
\quad \text{for } \ell<\kappa
$$
and $\omega_\kappa =0$.

To conclude the proof observe, that by Lemma~\ref{lemma:sum} we have
\begin{align*}
 |S_h(N)|&\leq \frac{1}{2^{2s}}\sum_{n=0}^{N-1}\left|\sum_{x,y=1}^{2^{s}}  \e\(\frac{h}{2^t}  u_{n + \tau_s xy} \)  \right| +2\tau_s 2^{2s}\\
 & 
 \leq \frac{1}{2^{2s}}\sum_{n=0}^{N-1}\left|\sum_{x,y=1}^{2^{s}}  \e\( \frac{h}{2^t}\cdot  \frac{g^{-n}}{g^{\tau_s xy}-bg^{-n}} \)  \right| + 2^{3s}\\
 & 
 \leq \frac{1}{2^{2s}}\sum_{n=0}^{N-1}\left|\sum_{x,y=1}^{2^{s}}  \e \(
 \frac{h g^{-n}(a_12^{s}xy+\ldots+a_\kappa2^{\kappa s} (xy)^\kappa)}{\kappa !2^t} 
 \)  \right| + N^{3/4}, \\ 
\end{align*}
where the coefficients $a_\ell=a_\ell(bg^{-n})$ for $\ell=1,\ldots, \kappa$, are determined as above with $bg^{-n}$ instead of $b$.

Write
$$
 \frac{hg^{-n}a_\ell 2^{\ell s}}{\kappa!2^t}=\frac{r_\ell}{q_\ell}, \quad \gcd (r_\ell,q_\ell)=1, \quad \ell=1,\ldots, \kappa, 
$$
with
\begin{equation}\label{eq:q}
 2^{t-\ell s-\omega_\ell}\leq q_\ell\leq \kappa! 2^{t-\ell s-\omega_\ell} \quad \ell=1,\ldots, \kappa.
\end{equation}

Then
\begin{equation}\label{eq:S(K)}
 |S_h(N)|\leq   \frac{1}{2^{2s}}\sum_{n=0}^{N-1}\left|\sum_{x,y=1}^{2^{s}}  \e\(
 f_n(x,y) \)  \right| + N^{3/4}, 
\end{equation}
where
$$
f_n(x,y)=\frac{r_1}{q_1}xy+\ldots+\frac{r_\kappa}{q_\kappa}(xy)^\kappa.
$$

Put
$$
\sigma_n=\sum_{x,y=1}^{2^{s}}  \e\( f_n(x,y) \).
$$

For $\kappa$, there exists a $k\in[2\kappa^2,4\kappa^2]$ such that for $N_{k,\kappa}$ we have the bound of Lemma~\ref{lem:ford}
(with $\kappa$ instead of $n$).

Then by Lemma~\ref{lem:Kor} we have
\begin{equation}\label{eq:s_n}
 \begin{split}
|\sigma_n|^{2k^2} \leq \left(64 k^2 \log(3Q)\right)^{\kappa/2} & 2^{(4k^2-2k)s} N_{k,\kappa}(2^s)\\
& \prod_{\ell=1}^{\kappa}\min\left\{2^{\ell s},\sqrt{q_{\ell}}+\frac{2^{\ell s}}{\sqrt{q_\ell}} \right\},
\end{split}
\end{equation}
where  by~\eqref{eq:q} we have $Q \leq \kappa! 2^t$ and thus
\begin{equation}\label{eq:Q}
\log (3Q) \leq \log(3\kappa! 2^t )\leq  t\kappa   \log(6\kappa).
\end{equation}

By the choice of $\kappa$ we have $s\kappa<t\leq s(\kappa+1)$.
As $\omega_\ell \leq \kappa \leq s$, under
$$
\frac{\kappa+1}{2}\leq \ell < \kappa
$$
we have by~\eqref{eq:q}
$$
q_\ell\leq \kappa!2^{s(\kappa+1-\ell)}\leq \kappa!2^{\ell s} \quad \text{and} \quad q_\ell >2^{s(\kappa-1-\ell)}
$$
thus
$$
\frac{1}{\sqrt{q_\ell}}+\frac{\sqrt{q_\ell}}{2^{\ell s}}\leq \frac{1+\kappa!}{\sqrt{q_\ell}}\leq \kappa^\kappa 2^{-\frac{s}{2}(\kappa-1-\ell)}.
$$
Whence
\begin{equation} 
 \begin{split}
 \label{eq:prod}
   \prod_{\ell=1}^{\kappa}\min\left\{2^{\ell s},\sqrt{q_{\ell}}+\frac{2^{\ell s}}{\sqrt{q_\ell }} \right\}&=2^{s\kappa(\kappa+1)/2} \prod_{\ell=1}^{\kappa}\min\left\{1,\frac{1}{\sqrt{q_\ell}}+\frac{\sqrt{q_\ell}}{2^{\ell s}}\right\}\\
 &   \leq  2^{s\kappa(\kappa+1)/2} \prod_{\frac{\kappa}{2}<\ell<\kappa}\kappa^\kappa 2^{-s (\kappa-1-\ell)/2}\\
 & \leq \kappa^{\kappa^2}2^{s\kappa(\kappa+1)/2- s(\kappa-2)(\kappa-4)/16}.
 \end{split}
\end{equation}
By Lemma~\ref{lem:ford} we have
\begin{equation} 
\label{eq:N(P)}
 N_{k,\kappa}(2^s) \leq \kappa^{3\kappa ^3} 2^{2ks-0.499\kappa ^2 s }. 
\end{equation}
Combining~\eqref{eq:s_n}, \eqref{eq:Q}, \eqref{eq:prod} and~\eqref{eq:N(P)},  we have
\begin{align*}
%% |\sigma_n|^{2k^2}&\leq \left(65 t \kappa^3 \log(2\kappa) \right)^{\kappa/2}
 |\sigma_n|^{2k^2} &\leq \left(64 t  k^3 \log(6\kappa) \right)^{\kappa/2}
  \kappa^{4\kappa^3} 2^{4k^2s +s\kappa(\kappa +1)/2-
 s (\kappa-2)(\kappa -4)/16-0.499\kappa^2 s}
\end{align*}
and therefore
$$
 |\sigma_n|\ll t^{1/(16 \kappa ^3)} 2^{2s-  s/(32770 \kappa^2)}.
$$

Since $t \kappa^2< (\frac{t}{s})^3s<(6r)^3s$, then
$$
2^{s/\kappa^2}=N^{rs/(t\kappa^2)}>N^{1/(216r^2)}.
$$
Moreover
$$
t^{1/\kappa^{3}}\leq N^{\log t/(129  r^2\log N)}\leq N^{\log \log N/(387 r^2\log N )}, 
$$
whence
$$
 |\sigma_n|\ll  2^{2s} N^{-\eta \rho^2},
$$
for some $\eta>0$ if $N$ is large enough. Thus by~\eqref{eq:S(K)} we have
$$
|S_h(N)|\leq  \frac{1}{2^{2s}}\sum_{n=0}^{N-1}|\sigma_n|+ N^{3/4} \ll N^{1-\eta \rho^2 }+ N^{3/4}\ll N^{1-\eta/r^2}
$$
which gives the result for $r\geq 129$.

If $r<129$, define 
$$
N_0=\fl{2^{t/129}} \qquad \rho_0 = \frac{\log N_0}{t} = \frac{\log 2}{129} + O(1/t).
$$
As $N\leq \tau <2^t$, we have
\begin{equation}\label{eq:M}
\frac{\log N_0}{\log N}> \frac{1}{129}.
\end{equation}
Then
$$
|S_h(N)|\leq \sum_{0\leq k < N/N_0} \left|\sum_{n=kN_0}^{(k+1)N_0-1} \e(hu_n/2^t)\right|.
$$
Applying the previous argument to the inner sums, we get
$$
|S_h(N)|\ll \frac{N}{N_0} N_0^{1-\eta  \rho_0^2}\ll N^{1-129^{-3} \eta  \rho_0^2}
$$
by~\eqref{eq:M}. Thus replacing $\eta$ to $\eta/129^3$, we conclude the proof.

\subsection{Proof of Theorem~\ref{thm:discrepancy}}

% 
% As in the proof of  Theorem~\ref{thm:main} we note the general question can be reduced to 
% case of $L=0$ and so we estimate $D(N) = D(0,N)$.

By the Erd\H{o}s-Tur\'an inequality, see~\cite{DrTi}
for any integer $H\geq 1$ we have
\begin{equation}\label{eq:ET}
 D(L,N)\ll \frac{1}{H} + \frac{2}{N}\sum_{h=1}^H\frac{1}{h}|S_h(L,N)|. 
\end{equation}
%where as before we denote $S_h(N) = S_h(0,N)$. 

Define
$$
H=\left\lfloor \frac{\tau_{t}}{\sqrt{N}} \right\rfloor, 
$$ 
where $\tau_t$ is as in Lemma~\ref{lemma:order}. 

For a given $1\leq h\leq H$, write $h=2^d j$ with odd $j$ and $d\leq \log_2 H$. Then consider the sequence $(u_n)$ modulo $2^{t-d}$. Then clearly
$$
S_h(L,N)=S_{d,j}(L,N).
$$
where $S_{d,j}(L,N)$ is defined as $S_{j}(L,N)$, however with respect to the modulus $ 2^{t-d}$.

By the above choice of parameters, we have
\begin{equation}\label{eq:t-d}
t-d\geq t-\log_2 H \geq  \frac{1}{2}\log_2 N +\beta > 17 \beta
\end{equation}
by Lemma~\ref{lemma:order}, thus
\begin{equation}\label{eq:tau_t-d}
\tau_{t-d}=2^{t-d-\beta+1}.
\end{equation}

Using~\eqref{eq:ET},  we have
\begin{equation}\label{eq:D}
\begin{split}
 D(L,N)& \ll \frac{1}{H} + \frac{1}{N}\sum_{h=1}^H \frac{1}{h}|S_h(L,N)|\\
 &\ll \frac{1}{H}+\frac{1}{N}\sum_{0\leq d \leq \log_2 H}\frac{1}{2^d}
 \sum_{\substack{1\leq j \leq H/2^d\\j \text{ odd}}}\frac{1}{j}|S_{d,j}(L,N)|. 
\end{split}
\end{equation}
For fixed $d$ and $j$ put 
$$
N_d=\left \lceil \frac{N}{\tau_{t-d}}\right\rceil \mand K_d = N-N_d \tau_{t-d} .
$$
Then
\begin{equation}\label{eq:split}
\begin{split}
|S_{d,j}(L,N)|& \leq \sum_{i=0}^{N_d-2} |S_{d,j}(L+i\tau_{t-d} , \tau_{t-d}) |\\
& \qquad \qquad + |S_{d,j}(L+(N_d-1)\tau_{t-d}  ,K_d)| . 
\end{split} 
\end{equation}

If $K_d<2^{8\beta}$, we use the trivial estimate 
$$
|S_{d,j}(L+(N_d-1)\tau_{t-d}  ,K_d)| \leq K_d<2^{8\beta}.
$$
As 
$$
8\beta <\frac{1}{2}(t-d-\beta) 
$$
by~\eqref{eq:t-d},
we get
\begin{equation}\label{eq:K_d}
|S_{d,j}(L+(N_d-1)\tau_{t-d}  ,K_d)| \leq \tau_{t-d}^{1- \eta  (t-d)^{-2}\(\log \tau_{t-d}\)^2} .
\end{equation}
If $K_d\geq 2^{8\beta}$, then  as $K_d\leq \tau_{t-d}$ we also have~\eqref{eq:K_d} by Theorem~\ref{thm:main}. Thus by~\eqref{eq:split} we have
$$
|S_{d,j}(L,N)|\ll N_d \cdot \tau_{t-d}^{1 -\eta  (t-d)^{-2}\(\log \tau_{t-d}\)^2} \ll 
N^{1- \eta  (t-d)^{-2}\(\log \tau_{t-d}\)^2 / \log N }.
$$
By~\eqref{eq:t-d} and~\eqref{eq:tau_t-d}  we have
$$
\frac{(\log \tau_{t-d})^3}{\log N (t-d)^2}=\frac{(t-d-\beta)^3}{\log N (t-d)^2}\geq \frac{(t-d-\beta)^3}{\log N t^2}\geq \frac{1}{8}\left(\frac{\log N}{t}\right)^2 = \rho^2/8 ,  
$$
whence
$$
|S_{d,j}(L,N)|\ll N^{1 - \eta \rho^2/8}.
$$
Then by~\eqref{eq:D},
\begin{align*}
 D(L,N)& \ll \frac{1}{H}+\sum_{0\leq d \leq \log_2 H}\frac{1}{2^d}\sum_{\substack{1\leq j \leq H/2^d\\j \text{ odd}}}\frac{1}{j}N^{-\eta \rho^2/8}\\
    & \ll 2^{-(t-\beta)/2}+N^{-\eta \rho^2/8 } \log H \ll \frac{1}{t}+N^{-\eta \rho^2/8 } \log H  \ll N^{-\eta \rho^2/16}
\end{align*}
if $N$ is large enough.

\section{Comments}

We note that an extension of our results to the case of sequences~\eqref{eq:seq} modulo prime powers $p^t$ 
with a prime $p\ge 3$ is immediate and can be achieved at the cost of merely typographical changes. 

We also note that all implied constants  are effective and can be evaluated (however at the cost of some additional 
technical clutter). 

It is certainly natural to study the multidimensional distribution of  the sequence generated by~\eqref{eq:Traj}, 
that is, the $s$-dimensional  vectors 
$$
(u_n, \ldots, u_{n+s-1}), \qquad n =1, \ldots, N.
$$
Our method is capable of addressing this problem, however investigating the $2$-divisibility of the corresponding 
polynomial coefficients which is an important part of our argument in Section~\ref{sec:proof exp} is 
more difficult and may require new arguments. 

We also use this as an opportunity to pose a question about studying short segments of the inversive generator 
modulo a large prime $p$. While results of Bourgain~\cite{Bou1,Bou2} give  a non-trivial bound
on exponential sums  for very short segments of  sequence
$ag^n \bmod p $, $n=1, \ldots, N$,  see also~\cite[Corollary~4.2]{Gar}, their analogues for even the 
simplest rational expressions like $1/(g^n-b)  \bmod p $ are not known. 
Obtaining such 
results beyond the standard range $N \ge p^{1/2+\varepsilon}$ (with any fixed $\varepsilon>0$) is apparently a 
difficult question requiring new ideas. 

\section*{Acknowledgement}
During the preparation of this  wok L.~M. was partially supported by the Austrian Science Fund FWF Projects P30405 
and  I.~S. by the Australian Research Council Grants DP170100786 and DP180100201.

\end{document}